\documentclass[12pt,a4paper]{amsart}
\usepackage[top=35mm, bottom=35mm, left=30mm, right=30mm]{geometry}
\usepackage{mathptmx}
\usepackage{mathrsfs}
\usepackage{amscd}
\usepackage{verbatim}
\usepackage{url}
\usepackage[all]{xy}
\usepackage{color}
\usepackage[colorlinks=true,citecolor=blue]{hyperref}
\usepackage{tikz}
\usetikzlibrary{arrows}
\usetikzlibrary{patterns}
\usetikzlibrary{lindenmayersystems,backgrounds}

\theoremstyle{plain}

\newtheorem{thm}{Theorem}[section]
\newtheorem{lem}[thm]{Lemma}
\newtheorem{prop}[thm]{Proposition}

\theoremstyle{definition}

\newtheorem{ques}{Question}

\newtheorem{exa}[thm]{Example}

\begin{document}
\title{Rigidity for higher rank lattice actions on dendrites}

\author[E.~Shi]{Enhui Shi}
\address[E. Shi]{Soochow University, Suzhou, Jiangsu 215006, China}
\email{ehshi@suda.edu.cn}

\author[H.~Xu]{Hui Xu}
\address[H. Xu]{CAS Wu Wen-Tsun Key Laboratory of Mathematics, University of Science and
Technology of China, Hefei, Anhui 230026, China}
\email{huixu2734@ustc.edu.cn}




\maketitle


\begin{abstract}
We study the rigidity in the sense of Zimmer for higher rank lattice actions on dendrites and show that: (1) if $\Gamma$ is a higher rank lattice and  $X$ is a nondegenerate dendrite with no infinite order points, then any action of $\Gamma$ on $X$ cannot be almost free; (2) if $\Gamma$ is further a finite index subgroup of $SL_n(\mathbb Z)$ with $n\geq 3$, then every action of $\Gamma$ on $X$ has a nontrivial almost finite subsystem.  During the proof, we get a new characterization of the left-orderability of a finitely generated
group  through its actions on dendrites.
\end{abstract}

\pagestyle{myheadings} \markboth{E. Shi and H. Xu} {Higher rank lattice actions }

\section{Introduction}

By a {\it higher rank lattice}, we mean a lattice in a connected real simple Lie group with finite center and with $\mathbb{R}$-rank at least two.
All group actions are assumed to be continuous. The one-dimensional Zimmer's rigidity conjecture says that every action on the circle by a higher rank lattice must
factor through a finite group action; this is equivalent to saying that every orbit of such actions is finite. Burger-Monod, and Ghys proved independently
the existence of finite orbits for higher rank lattice actions on the circle (\cite{Bu, Gh99}). This translates the conjecture into an equivalent form: no
higher rank lattice is left-orderable (a group is {\it left-orderable} if it admits a total ordering which is invariant by left translations).
The latter  was solved positively by Witte-Morris  for finite index subgroups of $SL_{n}(\mathbb Z)$ with $n\geq 3$ and by Deroin-Hurtado
for any higher rank lattice (\cite{Wi2, DH}).
\medskip

A {\it dendrite} is a Peano curve containing no simple closed curves. Group actions on dendrites are closely related
to the study of $3$-dimensional hyperbolic geometry (see e.g. \cite{Bo, Mi}). From the definition, we see that dendrites and the circle are
opposite to each other in topologies. So, it is interesting to compare the similarities and differences between group actions on these two
classes of curves. Similar to the previous results by Burger-Monod and Ghys, Duchesne-Monod proved the existence of finite orbits for every
higher rank lattice action on dendrites (\cite{DM2}). However, the exact analogy to the rigidity results by Witte-Morris and Deroin-Hurtado mentioned above
do not hold anymore for higher rank group actions on dendrites; in fact, it is easy to see that every countable infinite group can act faithfully on
a starlike dendrite of infinite order; even if the dendrite $X$  we considered is assumed to be of finite order, there can still exist a faithful higher rank lattice action on it
(see Section \ref{examples} for examples). One may consult \cite{AN, DM1, GM, SY} for some related discussions around dendrite homeomorphism groups.

\medskip

Although the exact analogy to the Zimmer's rigidity does not hold for higher rank lattice actions on dendrites, the following theorems we obtained
indicate that such actions are very restrictive.
We call an action of group $G$ on a topological space $X$ {\it almost free} if
for every $g\in G\setminus\{ {\rm id}\}$, the fixed point set ${\rm Fix}(g)$ of $g$ is totally disconnected.

\begin{thm}\label{main theorem 1}
Let $\Gamma$ be a higher rank lattice and $X$ be a nondegenerate dendrite without infinite order points. Then any $\Gamma$ action
on $X$ cannot be almost free.
\end{thm}

\medskip

Actually, Theorem \ref{main theorem 1} holds for any irreducible lattices of a real connected semisimple Lie group with finite center and $\mathbb{R}$-rank$\geq 2$. When we consider some specified higher rank lattices, a stronger form of rigidity can be shown. We introduce the notion of almost finite action in Section 2
to describe such rigidity. Roughly speaking, an almost finite action is a kind of limit of finite actions.

\begin{thm}\label{main theorem 2}
Let $X$ be a nondegenerate dendrite without infinite order points and $\Gamma$ be a finite index subgroup of $SL_{n}(\mathbb{Z})$ with $n\geq 3$. Then any action of $\Gamma$ on $X$ admits a $\Gamma$-invariant nondegenerate  subdendrite $Y$ such that the subsystem $(Y, \Gamma|Y)$ is almost finite  and the first point map $r: X\rightarrow Y$ is a contractible extension.
\end{thm}
\medskip

The proof of Theorem \ref{main theorem 1} relies on the non-left-orderability of higher rank lattices by Deroin-Hurtado, the elementarity
of higher rank lattice actions on dendrites by Duchesne-Monod, and an equivalent characterization of the left-orderability of a finitely generated
group obtained by the authors through its actions on dendrites. The proof of Theorem \ref{main theorem 2} is inspired by some key ideas of Witte-Morris in
\cite{Wi2}. To explain the necessity of the conditions in these theorems, we construct some examples in Section \ref{examples}.

\section{Preliminaries}

In this section, we will recall some notions and results around group actions, dendrites, and left-orderability of a countable group.
Particularly, we will introduce the dynamical realization technique which is very useful in proving the left-orderability of a group.

\subsection{Group actions} Let $G$ be a group and $X$ be a topological space. Recall that an {\it action} of $G$ on $X$
is a group homomorphism $\phi:G\rightarrow {\rm Homeo}(X)$, where ${\rm Homeo}(X)$ denotes the homeomorphism group of $X$; we use the pair $(X, G)$ to denote this action and use $gx$ or $g.x$
to denote $\phi(g)(x)$ for $g\in G$ and $x\in X$. If $\ker(\phi)=e_G$, we say the action $(X, G)$ is {\it faithful};
if $\ker(\phi)$ is finite, we say the action is {\it almost faithful}. We also call the action $(X, G)$ a {\it system}. For $x\in X$, the set
$Gx:=\{gx:g\in G\}$ is called the {\it orbit} of $x$; a subset $A$ of $X$ is
called $G$-{\it invariant} if $Gx\subset A$ for any $x\in A$; if $Gx=\{x\}$, then $x$ is called a {\it fixed point} of $G$; we use ${\rm Fix}(G)$ to
denote the fixed point set of $G$.
Suppose $A$ is closed and $G$-invariant. Then we naturally have a restriction action of $G$ to $A$, which is denoted by $(A, G|A)$ and is called a {\it subsystem} of $(X, G)$.
If $\phi(G)$ is finite, we say the action $(X, G)$ is {\it finite}. It follows from the Margulis' normal subgroup theorem (\cite[Chapter IV]{Mar})
that every higher rank lattice action either is finite or is almost faithful. For two actions $(X, G)$ and $(Y, G)$, if there is a continuous
surjection $\phi:X\rightarrow Y$ with $\phi(gx)=g\phi(x)$ for each $g\in G$ and $x\in X$, then $\phi$ is said to be a {\it factor map}
and $(Y, G)$ is said to be a {\it factor} of $(X, G)$  or $(X,G)$ is an {\it extension} of $(Y,G)$; if $\phi$ is additionally a homeomorphism, then $\phi$ is called
a {\it topological conjugation}, and $(X, G)$, $(Y, G)$ are called {\it topologically conjugate}.
 An extension $\phi: (X,G)\rightarrow (Y,G)$ is a {\it contractible extension} if for each $y\in Y$, $\phi^{-1}(y)$ is contractible; that is,
there is a sequence $g_1, g_2,\cdots \in G$ such that ${\rm diam}(g_i\phi^{-1}(y))\rightarrow 0$ as $i\rightarrow \infty$.

\medskip
If $(X_i, G), i=0, 1, 2, \cdots,$ is a sequence
of $G$ actions associated to each $i$ a factor map $\phi_i:X_{i+1}\rightarrow X_i$, then we say that these $(X_i, G)$ together with $\phi_i$'s  form
an {\it inverse system} and call each $\phi_i$ a {\it bonding map}. The inverse limit of this inverse system is defined to the set
\[\underset{\longleftarrow}{\lim}(X_i, G):=\left\{(x_0,x_1,\cdots)\in\prod_{i=0}^{\infty} X_i: \phi_i(x_{i+1})=x_i, \text{ for each }i\right\}\]
together with a specified action by $G$: $g.(x_0, x_1, \cdots)=(gx_0, gx_1, \cdots)$ for each $g\in G$; we use $(\underset{\longleftarrow}{\lim}(X_i, G), G)$
to denote this specified action. It is known that
if each $X_i$ is a compact metric space, then so is $\underset{\longleftarrow}{\lim}(X_i, G)$. We call a group action
$(X, G)$ being {\it almost finite} if it is topologically conjugate to an inverse limit $(\underset{\longleftarrow}{\lim}(X_i, G), G)$
with each $(X_i, G)$ being a finite action.

\subsection{Dendrites} By a {\it continuum} we mean a connected compact metrizable space. A continuum is {\it nondegenerate} if it is not a single point.
An (nondegerate) {\it arc} is a continuum homeomorphic to the closed interval $[0, 1]$. A {\it dendrite} is a continuum that is locally connected and contains no simple closed curve.  In the case of a dendrite $X$, the  Menger-Urysohn order ({\it order} for short) of a point $x\in X$  is just the cardinality of the set of connected components of $X\setminus\{x\}$. A point of $X$ is an {\it end point}, {\it regular point}, and {\it branch point} if its order is one, two, and $\geq 3$, respectively.
If the orders of all points in $X$ have an uniform upper bound, then we say $X$ is of {\it finite order}.
For $a, b\in X$, we use $[a, b]$ to denote the unique arc (may be degenerate) connecting $a$ and $b$; and use
$[a, b), (a, b], (a, b)$ to denote the sets $[a, b]\setminus\{b\}, [a, b]\setminus \{a\}, [a, b]\setminus \{a, b\}$ respectively. It is known that
the end point set of a nondegenerate dendrite $X$ is nonempty; the regular point set of $X$ is dense and uncountable. If $Y$ is a
subdendrite of $X$, then for each $x\in X$, there is a unique $r(x)\in Y$ with $[x, r(x)]\cap Y=\{r(x)\}$; the map $r$ so defined
is called the {\it first point map} from $X$ to $Y$. One may refer to
\cite{Nad} for more details around dendrites.

\begin{lem} \label{fixed point}
Let $X$ a nondegenerate dendrite and $h:X\rightarrow X$ be a homeomorphism. If $h$ fixes an end point $e\in X$, then $h$ fixes another point $o\not=e$.
\end{lem}
\begin{proof}
Take a point $u\not=e\in X$. Since $e$ is an end point, $h([e, u])\cap [e, u]=[e, v]$ for some $v\not=e\in X$.
Then there is $w\in [e, v]$ such that either $h(w)=v$ or $h^{-1}(w)=v$. WLOG, we assume that $h(w)=v$.
Then $[e, w]\subset [e, h(w)]\subset [e, h^2(w)]\subset\cdots$. Let $o=\lim\limits_{i\rightarrow\infty} h^i(w)$. Then $o\not=e$ and $h(o)=o$.
\end{proof}

The following proposition will be used later.

\begin{prop}\label{fixed arcs}
Let $G$ be a finitely generated nilpotent group acting on a nondegenerate dendrite $X$. Suppose that $G$ fixes an end point $z$ of $X$. Then $G$ has another fixed point.
\end{prop}
\begin{proof}
Take a subnormal series $G=G_0\trianglerighteq G_1\trianglerighteq\cdots \trianglerighteq G_{n}\trianglerighteq G_{n+1}=\{e\}$ such that $G_{i}/G_{i+1}$ is cyclic for each $i$. Take $g_i\in G_{i}\setminus G_{i+1}$ such that $G_{i}/G_{i+1}\cong \langle g_iG_{i+1}\rangle$. Thus $G_{i}=\langle g_i, \cdots, g_n\rangle$ for each $i=0,1,\cdots,n$.

\medskip
By Lemma \ref{fixed point}, there is a point in $X\setminus\{z\}$ which is fixed by $G_n$. Now assume that there is a point $x\in X\setminus\{z\}$ fixed by $G_i$ for some $i\in\{1,\cdots,n\}$. Since $z$ is an end point, there is  $y\in X\setminus\{z\}$ such that $g_{i-1}([z,x])\cap[z,x]=[z,y]$.
Then there is some $w\in [z, y]$ such that either $g_{i-1}(w)=y$ or $g_{i-1}^{-1}(w)=y$.  WLOG, we may assume that $g_{i-1}(w)=y$. Then $[z,w]\subset [z, g_{i-1}(w)]\subset [z, g_{i-1}^{2}(w)]\subset\cdots$.   Let $u=\lim\limits_{k\rightarrow \infty} g_{i-1}^{k}(w)$. Then $u$ is fixed by $g_{i-1}$.

\medskip
We claim that $u$ is fixed by $G_i$ and hence is fixed by $G_{i-1}$. We may assume that $g_{i-1}(x)\neq x$, otherwise $u=y=x$ is fixed by $G_i$. Since $g_{i-1}$ normalizes $G_{i}$, the image of any $G_i$-fixed point  under $g_{i-1}$ is also fixed by $G_i$.  By the definition of $y$, it is the point that  the arc $[z, g_{i-1}(x)]$ branching away from the arc $[z, x]$.  Now that $z, x, g_{i-1}(x)$ are fixed by $G_i$, so is $y$. Further, each $g_{i-1}^{k}(w)$ is fixed by $G_i$ and hence $u$ is fixed by $G_i$.

\medskip
By induction, there is some point in $X\setminus\{z\}$ fixed by $G$.
\end{proof}

The following theorem is due to Duchesne-Monod (\cite{DM2}).

\begin{thm} \label{higher rank dendrite}
Let $\Gamma$ be a higher rank lattice acting on a dendrite $X$. Then either $\Gamma$ has a fixed point or has an invariant  arc.
\end{thm}

\subsection{Ordering spaces} \label{ordering}
Let $\Gamma$  be a countable group and let $\Delta=\{(\gamma,\gamma):\gamma\in \Gamma\}$.  A total order relation $\preceq$ on $\Gamma$ corresponds to a unique point $\varphi\in\{-1, 1\}^{\Gamma\times\Gamma\setminus \Delta}$ satisfying
\begin{itemize}
\item[(R)] (Reflexivity) $\varphi(g,h)=-\varphi(h,g)$;
\item[(T)] (Transitivity) if $\varphi(f,g)=\varphi(g,h)=1$, then $\varphi(f,h)=1$;
\end{itemize}
by setting $\varphi(g,h)=1$ whenever $g\succ h$. Then the set $\mathcal{O}(\Gamma)$ of total orders on $\Gamma$ corresponds to a subset of $\{-1, 1\}^{\Gamma\times\Gamma\setminus \Delta}$. Taking the discrete topology on $\{-1, 1\}$ and endowing $\{-1, 1\}^{\Gamma\times\Gamma\setminus \Delta}$ with the product topology, the subset consisting of $\varphi$ satisfying (R) and (T) is  closed; this leads to a compact metrizable topology on $\mathcal{O}(\Gamma)$.  Furthermore, for $\varphi\in\mathcal{O}(\Gamma)$, if it satisfies additionally that
\begin{itemize}
\item[(L)] (Left-invariance) $\varphi(fg,fh)=\varphi(g,h)$ for any $f,g,h\in \Gamma$ with $g\neq h$,
\end{itemize}
then $\varphi$ is said to be a {\it left-ordering} on $\Gamma$. According to the condition $(L)$, the space $\mathcal{LO}(\Gamma)$ of left-orderings on $\Gamma$ forms a closed subspace of $\mathcal{O}(\Gamma)$. We say $\Gamma$ is {\it left-orderable} if  $\mathcal{LO}(\Gamma)\not=\emptyset$.
\medskip

The following theorem is due to Deroin-Hurtado (\cite{DH}).

\begin{thm} \label{non left order}
No higher rank lattice is left-orderable. (This is equivalent to saying that every orientation-preserving action
on $[0, 1]$ by a higher rank lattice is trivial.)
\end{thm}

\subsection{Dynamical realizations} \label{dynamical realization}
 The following proposition is a dynamical characterization of left-orderability (see e.g. \cite{Gh01, Na1}).
\begin{prop} \label{dyn real}
Let $\Gamma$ be a countable group. Then $\Gamma$ is left-orderable if and only if it admits a faithful action
on the real line $\mathbb R$ by orientation-preserving homeomorphisms.
\end{prop}
The proof of the necessity part of Proposition \ref{dyn real} uses the dynamical realization technique, which will be used later. So,
we outline the construction process here.

\medskip
Suppose $\Gamma$ is a left-orderable group with a left-ordering $\preceq$.
We enumerate $\Gamma$ as $\{g_i: i=1, 2, 3, ...\}$. Define a map $t:\Gamma\rightarrow\mathbb R$ by the induction process:
let $t(g_1)=0$ and suppose $t(g_1),...,t(g_n)$ have been defined; if $g_{n+1}$ is greater than (resp. smaller than)
$t(g_1),...,t(g_n)$, then let $t(g_{n+1})=\max\{t(g_1),...,t(g_n)\}+1$ (resp. $\min\{t(g_1),...,t(g_n)\}-1$);
if $g_{n+1}$ lies between $g_i, g_j$ and  $g_i, g_j$ are adjacent for some $i\not=j\in\{1,...,n\}$, then let
$t(g_{n+1})=(t(g_i)+t(g_j))/2$. For $g\in \Gamma$, define the action of $g$ on $t(\Gamma)$ by letting $gt(g')=t(gg')$
for each $g'\in G$; then extend this action to the closure $\overline {t(\Gamma)}$ and extend further to the whole line
by mapping affinely on the maximal intervals in $\mathbb R\setminus \overline {t(\Gamma)}$. Thus we obtain an orientation-preserving faithful action
of $\Gamma$ on $\mathbb R$. This construction process is called the {\it dynamical realization}.

\section{Local conditions for left-orderability}
\medskip
Let $\Gamma$ be a countable group and $B$ be a nonempty subset of $\Gamma$.  If $\varphi\in\{-1, 1\}^{\Gamma\times\Gamma\setminus \Delta}$ satisfying
\begin{itemize}
\item[(R$_B$)]   $\varphi(g,h)=-\varphi(h,g)$, for any $g,h\in B$ with $g\neq h$;
\item[(T$_B$)]  if $\varphi(f,g)=\varphi(g,h)=1$, then $\varphi(f,h)=1$, for any $f, g,h\in B$ with $f\neq g, f\neq h, g\neq h$;
\end{itemize}
then $\varphi$ defines a total ordering $\preceq_{\varphi}$ on $B$ by setting $g\succ_{\varphi} h$ if and only if $\varphi(g,h)=1$ for any $g\not=h\in B$. Then the set
\[ \mathcal{O}(\Gamma; B):=\left\{ \varphi\in\{-1, 1\}^{\Gamma\times\Gamma\setminus \Delta}: \varphi \text{ satisfies } R_{B}, T_B\right\}  \]
is closed in $\{-1, 1\}^{\Gamma\times\Gamma\setminus \Delta}$ with respect to the topology given in Section \ref{ordering}. If $(B_n)_{n=1}^{\infty}$ is an increasing sequence of subsets of $\Gamma$ with $\Gamma=\bigcup_{n=1}^{\infty} B_n$, then $ \bigcap_{n=1}^{\infty}\mathcal{O}(\Gamma; B_n)=\mathcal{O}(\Gamma)$.

\medskip
Now, let $F, B, B'$ be nonempty subsets  of $\Gamma$ with $B\subset B'$ and $FB\subset B'$, where $FB=\{fb: f\in F, b\in B\}$.  For $\varphi\in \mathcal{O}(\Gamma; B')$, we say $\varphi$ is $(F,B)$-{\it invariant}, if it satisfies
\begin{itemize}
\item[(L$_{F,B}$)]   $\varphi(fg,fh)=\varphi(g,h)$ for any $f\in F$ and $g\not=h\in B$.
\end{itemize}
Then the set
\[\mathcal{L}_{F}\mathcal{O}(\Gamma; B,B'):=\left\{ \varphi\in\{-1, 1\}^{\Gamma\times\Gamma\setminus \Delta}: \varphi \text{ satisfies } R_{B'}, T_{B'}, L_{F,B}\right\}  \]
is also closed in $\{-1, 1\}^{\Gamma\times\Gamma\setminus \Delta}$.

\medskip
Let $\mathcal{L}_{F}\mathcal{O}(\Gamma)$ denote the set of total orderings on $\Gamma$ that are invariant under the left translations of $F$, i.e.
\[ \mathcal{L}_{F}\mathcal{O}(\Gamma)=\left\{ \varphi\in\{-1, 1\}^{\Gamma\times\Gamma\setminus \Delta}: \varphi \text{ satisfies } R, T, L_{F,\Gamma}\right\}, \]
where the conditions R and T are as in Section 2.3.
It is clear that $\mathcal{L}_{F}\mathcal{O}(\Gamma)$ is also closed in $\{-1, 1\}^{\Gamma\times\Gamma\setminus \Delta}$ and
\[ \mathcal{LO}(\Gamma)=\bigcap\{ \mathcal{L}_{F}\mathcal{O}(\Gamma): F \text{ is a finite subset of }\Gamma\}.\]

If $(B_n)_{n=1}^{\infty}$ is a sequence of subsets of $\Gamma$ satisfying $FB_n\cup B_n\subset B_{n+1}$ for each $n\geq 1$ and $\Gamma=\bigcup_{n=1}^{\infty}B_n$, then
\[\mathcal{L}_{F}\mathcal{O}(\Gamma)=\bigcap_{n=1}^{\infty} \mathcal{L}_{F}\mathcal{O}(\Gamma; B_n,B_{n+1}).\]

\medskip
According to the above discussion, we have
\begin{lem}\label{finitely orderable}
Let $\Gamma$ be a countable group. Suppose $F$ is a nonempty finite subset of $\Gamma$ and $(B_n)_{n=1}^{\infty}$ is a sequence of subsets of $\Gamma$ satisfying $FB_n\cup B_n\subset B_{n+1}$ for each $n\geq 1$ and $\Gamma=\bigcup_{n=1}^{\infty}B_n$. If for each $n\geq1$, there is a total ordering $\preceq_{n}$ on $B_{n+1}$ that is $(F,B_n)$-invariant, i.e. $fg\preceq_n fh$ whenever $g\preceq_n h$ for any $f\in F$ and $g,h\in B_n$, then there is an $F$-invariant total ordering on $\Gamma$. Further, if, for each finite subset $E$ of $\Gamma$, there is an $E$-invariant total ordering on $\Gamma$, then $\Gamma$ is left-orderable.
\end{lem}

\section{Left-orderability and group actions on dendrites}

In this section, we will give an equivalent characterization of left-orderability for a finitely generated group through
its actions on dendrites. Then we complete the proof of the first main theorem by using this characterization.

\begin{prop} \label{char dendrite action}
Let $\Gamma$ be a finitely generated group. Then $\Gamma$ is left-orderable if and only if it admits an
almost free action on a nondegenerate dendrite with an end point fixed.
\end{prop}

\begin{proof}
$(\Longrightarrow)$ Let $\preceq$ be a left-ordering on $\Gamma$ and numerate $\Gamma$ as $\{g_1, g_2, g_3, \cdots\}$.
Then by the dynamical realization as in Section \ref{dynamical realization}, we get an orientation-preserving faithful action of $\Gamma$ on
$\mathbb R$.  Extending this action to the two points compactification $I:=\{-\infty\}\cup\mathbb R\cup \{+\infty\}$
by letting $-\infty$ and $+\infty$ fixed by each $g\in\Gamma$, we get an action of $\Gamma$ on the arc $I$.

We claim
that this extended action on $I$ is almost free. Otherwise, there is some $g\not=e\in \Gamma$ with the fixed point set
${\rm Fix(g)}$ of $g$  not totally disconnected; so, it contains a nondegenerate arc $J$. By the definition of
dynamical realization, we see that there is some maximal open interval $(a, b)$ of $\mathbb R\setminus \overline {t(\Gamma)}$
with $J\subset [a, b]$, where $t$ is as in Section 2.4. Since $g$ fixes $a$ and $b$, $\{a, b\}\cap t(\Gamma)=\emptyset$.
Thus $a$ and $b$ are accumulation points of $t(\Gamma)$. Take $g', g''\in \Gamma$ with $0<a-t(g')<(b-a)/3$ and $0<t(g'')-b<(b-a)/3$.
Suppose $g'=g_m$ and $g''=g_n$ for some indices $m, n$. Let $k=\max\{n, m\}$ and let $n', m'\in\{1, 2,\cdots, k\}$ be such that $t(g_{m'})$ is maximal in $\{t(g_1), t(g_2),\cdots, t(g_k)\}\cap(-\infty, a)$
and $t(g_{n'})$ is minimal in $\{t(g_1), t(g_2),\cdots, t(g_k)\}\cap(b, +\infty)$ respectively.
Let $i$ be the first index so that $i>\max\{m', n'\}$ and $g_{m'}\prec g_i\prec g_{n'}$.
Then $t(g_i)=(t(g_{m'})+t(g_{n'}))/2\in (a, b)$, which is a contradiction.

Thus the action of $\Gamma$ on the arc
$I$ is almost free and $I$ is also a nondegenerate dendrite with an end point fixed by each $g\in\Gamma$.

\medskip

$(\Longleftarrow)$ Let $z\in X$ be an end point of $X$ fixed by $\Gamma$. Let $\{g_1^{\pm}, \cdots,g_{k}^{\pm}\}$ be a set of generators for $\Gamma$.
By Lemma \ref{fixed point}, for each $i\in\{1,\cdots,k\}$, there is a point $t_i\in X\setminus\{z\}$ fixed by $g_i$. Let $T$ be the smallest subcontinuum of $X$ containing $\{z,t_1,\cdots,t_n\}$, which is a subtree of $X$. Let $I$ denote the intersection of all arcs $[z, t_i]$, $i=1,\cdots, k$.  Since $z$ is an end point,
the arc $I$ is not reduced to a point. We write $I=[z, t]$ and give a canonical ordering $<$ on $[z,t]$ with $t>z$.

Fix a finite subset $F$ of $\Gamma$.  Choose a sequence $(B_n)_{n=1}^{\infty}$ of finite subsets of $\Gamma$ satisfying $FB_n\cup B_n\subset B_{n+1}$ for each $n\geq 1$ and $\Gamma=\cup_{n=1}^{\infty}B_n$. It is clear that such sequence exists.

 Given $n\geq 1$, there is a point $s\in (z,t)$ such that $g([z,s])\subset [z,t)$ for each $g\in B_{n+1}$. Now choose a dense sequence $(x_i)_{i=1}^{\infty}$ in $(z, s)$.  For each pair of two distinct $g,h\in B_{n+1}$, define $g\prec h$ if the smallest $j\geq 1$ for which $g(x_j)\neq h(x_j)$ is such that $g(x_j)<h(x_j)$ with respect to the canonical ordering $<$ on $[z, t]$. Indeed, such $j$ exists by the almost freeness of the $\Gamma$-action.  It is easy to verify that $\prec$ is a total ordering defined on $B_{n+1}$. Note that for each $g\in B_{n+1}$, the restriction of $f$ to $(z, s)$ is increasing with respect to $<$ on $[z,t]$. Thus for every $f\in F$ and $g,h\in B_n$, we have $fg\prec fh$ whenever $g\prec h$. Hence the ordering $\preceq$ on $B_{n+1}$ is $(F,B_n)$-invariant.

According to Lemma \ref{finitely orderable},  we conclude that $G$ is left-orderable.

\end{proof}

Now, we are ready to prove the first main theorem of the paper.

\begin{proof}[Proof of Theorem \ref{main theorem 1}]
Assume to the contrary that the action $(X, \Gamma)$ is almost free.
According to Theorem \ref{higher rank dendrite}, $\Gamma$ either has a  fixed point or preserves an arc.
\medskip

We discuss into two cases:

\medskip
{\bf Case 1.} $\Gamma$ preserves a nondegenerate arc $I$.  Then there is a subgroup $\Gamma'$ of $\Gamma$ with index  at most two such that the restriction of $\Gamma'$ to $I$ preserves the orientation of $I$. Since $\Gamma'$ is also a higher rank lattice, its action on $I$ is trivial by Theorem \ref{non left order}.
This contradicts the assumption of almost freeness of the $\Gamma$ action.

\medskip
{\bf Case 2.}  $\Gamma$ has a fixed point $z\in X$.  Since the order of $z$ is finite, there are finitely many connected components of $X\setminus \{z\}$. Fix a connected component $C$ of $X\setminus \{z\}$. Let $\Gamma''$ be the subgroup of $\Gamma$ that preserves $C$, i.e.
\[\Gamma''=\{\gamma\in\Gamma: \gamma(C)=C\}.\]
Then $\Gamma''$ has finite index in $\Gamma$ and hence is also a higher rank lattice. Consider the restriction action of $\Gamma''$ to $Y=\overline{C}=C\cup\{z\}$.
Noting that $z$ is an end point of $Y$ fixed by $\Gamma''$ and $\Gamma''$ is finitely generated, $\Gamma''$ is left-orderable by Proposition \ref{char dendrite action}.
This contradicts Theorem \ref{non left order}.
\end{proof}

\section{Almost finiteness for actions by finite index subgroups of $SL_n(\mathbb Z)$ }

Let $G$ be a countable group. A {\it left quasi-order} $\preceq$ on $G$ is a binary relation on $G$ satisfying
\begin{itemize}
\item[(1)] for any $g,h\in G$, either $g\preceq h$ or $h\preceq g$;
\item[(2)] for any $f,g,h\in G$, if $f\preceq g$ and $g\preceq h$, then $f\preceq h$;
\item[(3)] for any $f,g,h\in G$, if $g\preceq h$, then $fg\preceq fh$.
\end{itemize}
Let $\preceq $ be a left quasi-order on $G$ and $g,h\in G$. We say $g\ll h$ if either $g^{k}\preceq h$ for all $k\in\mathbb{Z}$ or  $g^{k}\preceq h^{-1}$ for all $k\in\mathbb{Z}$. We write $g\succ h$ if $g\not\preceq h$.
By (1), there is no element $g\in G$ satisfying $g\succ g$.

The following lemma is similar to Lemma 3.2 in \cite{Wi2}.
\begin{lem}\label{ll}
Let $\preceq$ be a left quasi-order on a group $G$. If $a,b,c\in G$ satisfies $[a,b]=a^{-1}b^{-1}ab=c^{r}$ for some $r\in \mathbb{Z}\setminus\{0\}$ and $c$ commutes with both $a$ and $b$, then either $c\ll a$ or $c\ll b$.
\end{lem}
\begin{proof}
By the definition of the relation $\ll$, for any $g,h\in G$, $ g\ll h$ is equivalent to $g^{\pm 1}\ll h^{\pm 1}$. Thus we may assume that $e:=e_{G}\preceq a,b,c$ and $r>0$;  
and assume to the contrary that $c\not\ll a$ and $c\not\ll b$. Thus there are some $p,q\in\mathbb{Z}_{+}$ such that $c^{p}\succ a$ and $c^{q}\succ b$. According to the left invariance, we have
\[e\prec a^{-1}c^{p},\ \  e\prec b^{-1}c^{q}.\ \   \]
Noting that  $e\preceq a,\ \  e\preceq b,\ \  e\preceq c$, we have for sufficiently large positive integer $m$: 
\begin{eqnarray*} 
e&\prec& (b^{-1}c^{q})^{m}(a^{-1}c^{p})^{m} a^{m}b^{m}\\
&=& [b^{m}, a^{m}] c^{m(p+q)}\\
&=& c^{-m^2r+m(p+q)}\\
&\preceq& e.
\end{eqnarray*}
This is a contradiction.
\end{proof}

\begin{lem}\label{key lem n=3}
Suppose that $\Gamma$ is a finite index subgroup of $SL_{3}(\mathbb{Z})$. If $\Gamma$ acts on a nondegenerate dendrite $X$ and fixes an end point $z$, then there is a point $s\in X\setminus\{z\}$ such that the arc $[z, s]$ is fixed by $\Gamma$ pointwise.
\end{lem}
\begin{proof}
Since $\Gamma$ has finite index in $SL_{3}(\mathbb{Z})$, there is some positive integer $r$ such that $\Gamma$ contains the following six elements
\[a_1=\begin{bmatrix}1&r&0\\0&1&0\\0&0&1  \end{bmatrix},\ \  a_2=\begin{bmatrix}1&0&r\\0&1&0\\0&0&1  \end{bmatrix},\ \ a_3=\begin{bmatrix}1&0&0\\0&1&r\\0&0&1  \end{bmatrix},  \]
\[a_4=\begin{bmatrix}1&0&0\\r&1&0\\0&0&1  \end{bmatrix},\ \  a_5=\begin{bmatrix}1&0&0\\0&1&0\\r&0&1  \end{bmatrix},\ \ a_6=\begin{bmatrix}1&0&0\\0&1&0\\0&r&1  \end{bmatrix}.  \]

Let $\Gamma_i=\langle a_{i-1}, a_i, a_{i+1}\rangle$ for each $i\in\mathbb{Z}/6\mathbb{Z}$. A straightforward verification shows that $[a_i,a_{i+1}]=e$ and $[a_{i-1},a_{i+1}]=a_{i}^{\pm r}$ for each $i\in\mathbb{Z}/6\mathbb{Z}$. Now, by Lemma \ref{fixed arcs}, there is a point $s_i$ different from $z$ that is fixed by $\Gamma_i$ for each $i\in\mathbb{Z}/6\mathbb{Z}$.

\medskip
For each $i\in\mathbb{Z}/6\mathbb{Z}$, fix  a linear order $\leq_{[z,s_i]}$ on $[z,s_i]$ with $z<s_i$. For every $x\in (z, s_i)$, define a left quasi-order $\preceq_{x}^{(i)}$ on $\Gamma_i$ by setting
\[ \gamma_1\preceq_{x}^{(i)} \gamma_2\text{ if and only if  }  \gamma_1(x)\leq_{[z,s_i]} \gamma_2(x), \text{ for any }\gamma_1,\gamma_2\in \Gamma_i.\]
Set $I=\cap_{i=1}^{6}[z, s_i)$ and take a point $y\not=z\in I$ such that $a_j^{\pm 1}[z,y]\subset I$ for each $j\in\{1,\cdots,6\}$. By Lemma \ref{ll}, for each $i$,
\[ \text{either } a_i\ll_{y}^{(i)} a_{i-1} \text{ or } a_i\ll_{y}^{(i)} a_{i+1}.\]

 {\bf Claim.} $a_{i}(y)=y$ for each $i\in\{1,\cdots,6\}$.

\medskip

{\it Proof of the Claim.} To the contrary, we may assume that $a_1(y)\neq y$.
Note that either  $a_1\ll_{y}^{(1)} a_{6}$ or $a_1\ll_{y}^{(1)} a_{2}$. We discuss into two cases.

\medskip

{\bf Case 1}. $a_1\ll_{y}^{(1)} a_{2}$. Then, by definition, either $a_1^{k}\preceq_{y}^{(1)} a_2$ for all $k\in\mathbb{Z}$ or  $a_1^{k}\preceq_{y}^{(1)} a_2^{-1}$ for all $k\in\mathbb{Z}$. In either case, we have $a_{1}^{k}(y)\leq_{[z,s_1]} \max\{ a_{2}(y), a_{2}^{-1}(y)\}$, for all $k\in\mathbb{Z}$. If $a_2(y)= y$ then $a_1(y)=y$ as well;
this contradicts the assumption. Thus we have that $a_2(y)\neq y$ and hence $a_2(y)\neq a^{-1}_{2}(y)$. So $a_2\not\ll_{y}^{(2)} a_1 $. Further we have that $a_2\ll_{y}^{(2)} a_3$. Similarly, if $a_3(y)=y$ then $a_2(y)=y$, which implies that $a_1(y)=y$; this is a contradiction. Thus $a_3(y)\neq y$ and then $a_3\not\ll_{y}^{(3)} a_2$. Inductively, we have
\[ a_1\ll_{y}^{(1)} a_2\ll_{y}^{(2)} a_3\ll_{y}^{(3)} a_4\ll_{y}^{(4)} a_5\ll_{y}^{(5)} a_6\ll_{y}^{(6)} a_1.\]
By the definitions of these quasi-orders and the choice of $y$, we have
\begin{eqnarray*}
& &\sup_{k\in\mathbb{Z}}a_1^{k}(y)\leq_{I}\max\{a_2(y),a_2^{-1}(y)\} \leq_{I}\max\{a_3(y),a_3^{-1}(y)\}\leq_{I} \max\{a_4(y),a_4^{-1}(y)\} \\
& &\qquad \leq_{I}\max\{a_5(y),a_5^{-1}(y)\} \leq_{I}\max\{a_6(y),a_6^{-1}(y)\} \leq_{I}\max\{a_1(y),a_1^{-1}(y)\},
\end{eqnarray*}
where $\leq_{I}$ is the natural linear order on $I$ with respect to which $z$ is minimal. Thus we have $a_1(y)=y$, which contradicts our assumption.
So the claim holds in this case.

\medskip
{\bf Case 2}. $a_1\ll_{y}^{(1)} a_{6}$. Similar to Case 1,  we have
\[ a_1\ll_{y}^{(1)} a_6\ll_{y}^{(6)} a_5\ll_{y}^{(5)} a_4\ll_{y}^{(4)} a_3\ll_{y}^{(3)} a_2\ll_{y}^{(2)} a_1,\]
and hence
\begin{eqnarray*}
&& \sup_{k\in\mathbb{Z}}a_1^{k}(y)\leq_{I}\max\{a_6(y),a_6^{-1}(y)\} \leq_{I}\max\{a_5(y),a_5^{-1}(y)\}\leq_{I} \max\{a_4(y),a_4^{-1}(y)\} \\
&& \leq_{I}\max\{a_3(y),a_3^{-1}(y)\} \leq_{I}\max\{a_2(y),a_2^{-1}(y)\} \leq_{I}\max\{a_1(y),a_1^{-1}(y)\}.
\end{eqnarray*}
We also have $a_1(y)=y$ and the claim  holds in this case.

\medskip
Let $\Gamma'=\langle a_1,a_2,a_3,a_4,a_5,a_6\rangle$. By a result of Tits in \cite{Tits} (also refer to \cite{Men}), we know that $\Gamma'$ has finite index in $\Gamma$.  Note that the claim holds for any  $y\in I$ with $a_j^{\pm 1}[z,y]\subset I$ for each $j\in\{1,\cdots,6\}$. Thus there is a $t\in X\setminus\{z\}$ such that the arc $[z,t]$ is fixed by $\Gamma '$ pointwise. Now the set $\Gamma [z, t]=\{\gamma [z,t]: \gamma\in \Gamma\}$ consists of finitely many arcs. Since $z$ is an end point fixed by $\Gamma$, there is some $s\in X\setminus\{z\}$ with $[z,s]=\cap \Gamma [z, t]$; and $[z,s]$ is then fixed by $\Gamma$ pointwise.
\end{proof}

\begin{prop}\label{key lem}
Suppose that $\Gamma$ is a finite index subgroup of $SL_{n}(\mathbb{Z})$ with $n\geq 3$. If $\Gamma$ acts on a nondegenerate dendrite $X$ and fixes an end point $z$, then there is a point $s\in X\setminus\{z\}$ such that the arc $[z, s]$ is fixed by $\Gamma$ pointwise.
\end{prop}
\begin{proof}
Let $u_{i,j}$ be the matrix in $SL_{n}(\mathbb{Z})$ with $1$'s along the diagonal and at the entry $(i,j)$ and $0$'s elsewhere.  Since $\Gamma$ has finite index in $SL_{n}(\mathbb{Z})$, there is an $\ell\in\mathbb{Z}_{+}$ such that $u_{i,j}^{\ell}\in \Gamma$ for each $i,j\in\{1,\cdots,n\}, i\neq j$. Given $1\leq i<j\leq n-1$, let
\[ a_1=u_{i,j}^{\ell},\ \ a_2=u_{i,j+1}^{\ell},\ \ a_3=u_{j,j+1}^{\ell},\ \  a_4=u_{j,i}^{\ell},\ \  a_5=u_{j+1,i}^{\ell},\ \  a_6=u_{j+1,j}^{\ell}.\]
A straightforward check shows that they also satisfy  that $[a_i,a_{i+1}]=e$ and $[a_{i-1},a_{i+1}]=a_{i}^{\pm r}$ for each $i\in\mathbb{Z}/6\mathbb{Z}$.
Applying the proof of Lemma \ref{key lem n=3} to the group $\Gamma_{i,j}=\langle a_1,\cdots, a_6\rangle$, there is some $s_{i,j}\in X\setminus\{z\}$ such that $\Gamma_{i,j}$ fixes the arc $[z,s_{i,j}]$ pointwise. Let $t\in X\setminus\{z\}$ be such that $[z,t]=\bigcap_{1\leq i<j\leq n-1}[z,s_{i,j}]$. Then the arc $[z,t]$ is fixed pointwise by $\Gamma_{i,j}$, for each $1\leq i<j\leq n-1$. Now let $\Gamma'=\langle u_{i,j}^{\ell}: 1\leq i,j\leq n, i\neq j\rangle$. Then $[z,t]$ is fixed by $\Gamma'$ pointwise, by noting that $\Gamma'=\langle \Gamma_{i,j}: 1\leq i<j\leq n-1\rangle$. Recall that $\Gamma'$ has finite index in $SL_{n}(\mathbb{Z})$ and hence in $\Gamma$ (see \cite{Tits} or \cite{Men}).  Thus the set $\Gamma [z, t]=\{\gamma [z,t]: \gamma\in \Gamma\}$ consists of finitely many arcs. Since $z$ is an end point fixed by $\Gamma$, there is some $s\in X\setminus\{z\}$ such that $[z,s]=\cap\Gamma [z, t]$ and $[z,s]$ is fixed by $\Gamma$ pointwise.
\end{proof}

Now, we are ready to prove the second main theorem. Though the following two lemmas are known to experts in group theory, we afford the proofs for the convenience of the readers.

 \begin{lem}\label{finite index normal}
Let $H$ be a subgroup of $G$ of finite index. Then there is a normal subgroup $K$ of $G$ that is contained in $H$ and has finite index in $G$. 
\end{lem}
\begin{proof}
Assume that the index of $H$ in $G$ is $n>0$. Then the canonical action of $G$ on the coset space $G/H$ induces a homomorphism $\phi: G\rightarrow {\rm Sym}(n)$, where ${\rm Sym}(n)$ denotes the permutation group of $n$ elements. So $\ker(\phi)\leq H$ and $\ker(\phi)\trianglelefteq G$. Further, $[G:\ker(\phi)]\leq n!$. Thus $\ker(\phi)$ is just what we are looking for.
\end{proof}

\begin{lem} \label{SL}
Let $N$ be a positive integer and $(\Gamma_{\alpha})_{\alpha\in A}$ be a family of finite indexed normal subgroups of $SL_n(\mathbb{Z})$ with $n\geq 3$.  If  for each $\alpha\in A$, $[SL_{n}(\mathbb{Z}): \Gamma_{\alpha}]\leq N$, then the intersection $\bigcap_{\alpha\in A}\Gamma_{\alpha}$ also has finite index in $SL_{n}(\mathbb{Z})$.
\end{lem}
\begin{proof}
Let $\Gamma=SL_{n}(\mathbb{Z})$ and for each positive integer $k$, let $\Gamma(k)$ be the principal congruence subgroup of level $k$, i.e. the kernel of $SL_n(\mathbb{Z})\rightarrow SL_n(\mathbb{Z}/k\mathbb{Z})$.
Let $u_{i,j}$ be the matrix in $SL_{n}(\mathbb{Z})$ with $1$'s along the diagonal and at the entry $(i,j)$ and $0$'s elsewhere.

\medskip
Fix a $\Gamma_{\alpha}$.  Since $[SL_{n}(\mathbb{Z}): \Gamma_{\alpha}]\leq N$, there exists $\ell_{i,j}$ such that $e_{i,j}^{\ell_{i,j}}\in\Gamma_{\alpha}$ for each $i,j\in\{1,\cdots,N\}$. Let $\ell$ be the least common multiple of $\ell_{i,j}$. Thus $e_{i,j}^{\ell}\in\Gamma_{\alpha}$ for each $i,j$.  By the result of Mennicke in \cite{Men} that $\{e_{i,j}^{\ell}: 1\leq i,j\leq N\}$ generates $\Gamma(\ell^2)$, we have $ \Gamma(\ell^2)\subset\Gamma_{\alpha}$. Now let $k$ be the least common multiple of $\{1,\cdots,N\}$. Then we have $\Gamma(k^2)\subset \bigcap_{\alpha\in A}\Gamma_{\alpha}$, which implies the later one has finite index in $SL_{n}(\mathbb{Z})$.
\end{proof}

\begin{proof}[Proof of Theorem \ref{main theorem 2}]

We first claim that there is a  fixed point of $\Gamma$. If not, $\Gamma$  leaves an arc $I$ invariant  by Theorem \ref{higher rank dendrite}.
Then  there is $\gamma\in \Gamma$ and a subgroup $\Gamma'\leq \Gamma$ such $\Gamma=\Gamma'\cup \gamma\Gamma'$ and $\Gamma'$ preserves the orientation of $I$.
By Theorem \ref{non left order}, the restriction action of $\Gamma'$ to $I$ is trivial. Note that $\gamma$ has a fixed point $p$ in $I$, which is then a fixed point of
 $\Gamma$. Thus the claim holds.

\medskip

In the following, we will define, for each $i=0, 1, 2, \cdots,$ a subdendrite $X_i$ of $X$ and a normal subgroup $\Gamma_i$ of $\Gamma$ with finite index, satisfying
that:
\begin{itemize}
\item[(1)] $\{p\}=X_0\subset X_1\subset X_2\subset\cdots$, and $\Gamma=\Gamma_0\supset\Gamma_1\supset\Gamma_2\supset\cdots;$
\item[(2)] $X_i$ is contained in ${\rm Fix}(\Gamma_i)$ and is $\Gamma$-invariant, and $(X_i, \Gamma|X_i)$ is a finite action;
\item[(3)] the first point map $\phi_{i}: X_{i+1}\rightarrow X_i$ is a factor map from $(X_{i+1}, \Gamma|X_{i+1})$ to $(X_i, \Gamma|X_i)$;
\item[(4)] for any $i\in\{0,\cdots, m-1\}$, the image of $X\setminus X_{i+1}$ under the first point map $r_{i+1}: X\rightarrow X_{i+1}$ is contained in $X_{i+1}\setminus \{x\in X_{i}: {\rm ord}(x)\leq i+1\}$.
\end{itemize}
\medskip

First take a fixed point $p$ of $\Gamma$. Let $X_0=\{p\}$ and let $\Gamma_0=\Gamma$; then $X_0$ is $\Gamma_0$-invariant. Since the order of $p$ is finite, the number of connected components of $X\setminus\{p\}$
is finte; thus there is a finite index normal subgroup $\Gamma_1$ of $\Gamma$ which leaves each component of $X\setminus\{p\}$ invariant.
For each such component $C$, applying Proposition \ref{key lem} to $(C\cup\{p\}, \Gamma_1|C\cup\{p\})$, there is a point $s_{C}\in C$ such that the arc $[p, s_{C}]$ is fixed pointwise by  $\Gamma_1$ (note that $p$ is an end point of $C\cup\{p\}$).
\medskip

Let $X_1$ be the connected component of ${\rm Fix}(\Gamma_{1})$ that contains $p$. Then $X_1$ is a nondegenerate subdendrite of $X$.  Since $p$ is fixed by $\Gamma$ and $\Gamma_1$ is normal in $\Gamma$,  $X_1$ is $\Gamma$-invariant. Thus the action of $\Gamma$ on $X_1$ factors through $\Gamma/\Gamma_1$ action, which is finite.
It is clear that the first point map $\phi_0:X_1\rightarrow X_0$ satisfies $\phi_0(gx)=g\phi_0(x)=p$ for each $g\in G$ and $x\in X_1$.

\medskip
Now suppose that we have defined $X_0,\cdots, X_m$ and $\Gamma=\Gamma_0\trianglerighteq\Gamma_1\trianglerighteq\cdots\trianglerighteq\Gamma_m$, for some positive integer $m$, satisfying (1)-(4) above. Let $B_m=\{x\in X_m: {\rm ord}(x)\leq m+1\}$ and let $r_{m}: X\rightarrow X_m$ be the first point map.  For each point $x\in B_{m}$, the preimage $Y_{x}:=r_{m}^{-1}(\{x\})$ is a sub-dendrite of $X$ and has a unique intersecting point $x$ with $X_m$.
By assumption, every $x\in X_m$ is fixed by $\Gamma_m$. For each $x\in B_m$, the cardinality of the set ${\rm Comp}(Y_x-x)$  of components of $Y_x\setminus\{x\}$ is no greater than $m+1$. Thus there is a finite index subgroup $\Gamma_{m, x}$ of $\Gamma_m$ such that $\Gamma_{m,x}$ leaves each component of $Y_x\setminus\{x\}$ invariant.  By Lemma \ref{finite index normal}, for each $x\in B_m$, there is a subgroup $\Gamma_{m,x}'\leq \Gamma_{m,x}$ which is normal and has finite index in $\Gamma$. Applying Lemma \ref{SL}, the subgroup
$\Gamma_{m+1}=\bigcap_{x\in B_{m}}\Gamma_{m,x}'\leq \Gamma_m$ has finite index in $SL_{n}(\mathbb{Z})$ and is normal in $\Gamma$.

\medskip
By the definition of $\Gamma_{m+1}$, $X_m$ is contained in the fixed point set of $\Gamma_{m+1}$.  Let $X_{m+1}$ be the connected component of ${\rm Fix}(\Gamma_{m+1})$ that contains $X_{m}$.  Since $p\in X_{m+1}$ is fixed by $\Gamma$ and $\Gamma_{m+1}$ is normal in $\Gamma$, we have that $\Gamma$ leaves $X_{m+1}$ invariant.
By Proposition \ref{key lem},  $X_{m+1}$ satisfies (4). It is clear that (1)-(3) hold for $X_{m+1}$.

\medskip
In such way, we inductively defined the desired sequence $X_0\subset X_1\subset\cdots$ and $\Gamma=\Gamma_0\trianglerighteq\Gamma_1\trianglerighteq\cdots$ satisfying (1)-(4). Let $Y=\overline{\cup_{i=0}^\infty X_i}$. Then $Y$ is a $\Gamma$-invariant dendrite and topologically conjugate to the inverse limit $(\underset{\longleftarrow}{\lim}(X_i, \Gamma), \Gamma))$
by \cite[Theorem 10.36]{Nad}. Since each  $(X_i, \Gamma|X_i)$ is a finite action, $(Y, \Gamma)$ is almost finite.

\medskip

It remains to show that the first point map $r: X\rightarrow Y$ is a contractible extension. It is obvious in case of $Y=X$. So we suppose that $Y\neq X$.
\medskip

{\bf Claim 1.} For each $m$ and $x\in X_{m+1}\setminus X_m$, $X\setminus \{r_m(x)\}$ has at least two distinct connected components,
where $r_m:X\rightarrow X_m$ is the first point map.
\medskip

{\it Proof of Claim 1.} Since $x\in X_{m+1}\setminus X_m$, $[x, r_m(x))\cap {\rm Fix}(\Gamma_m)=\emptyset.$
This together with Proposition \ref{key lem} implies that there exists $\gamma\in\Gamma_m$ such that $\gamma[x, r_m(x))\cap [x, r_m(x))=\emptyset$. Then
the components of $X\setminus\{r_m(x)\}$ containing $x$ and $r_m(x)$ respectively are distinct.
\medskip

{\bf Claim 2.} For any $x\in X\setminus Y$, the orbit of $r(x)$ is infinite.
\medskip

{\it Proof of Claim 2.} By the definition of $X_m$ and Proposition \ref{key lem}, we see that $r(x)\in Y\setminus \cup_{i=0}^\infty X_i$.
From Property (4), there is a subsequence $(m_j)$ with $\gamma_{m_{j+1}}(x)\in X_{m_{j+1}}\setminus X_{m_j},$ where $r_{m_j}:X\rightarrow X_{m_j}$
is the first point map. It follows from Claim 1 that there are $g_{m_j}\in\Gamma_{m_j}$ and $C_{m_j}\not=C_{m_j}'\in {\rm Comp}(X\setminus \{r_{m_j}(x)\})$
such that $r_{m_{j+1}}(x)\in C_{m_j}$ and $g_{m_j}(r_{m_{j+1}}(x))\in C_{m_j}'$. Note that $C_{m_1}\supset C_{m_2}\supset C_{m_3}\supset\cdots$
and $C_{m_j}'\subset C_{m_{j-1}}$ for each $j$. Thus we  have
\begin{eqnarray*}
&r_{m_{j+1}}(x)\in C_{m_j},\\
&g_{m_j}r_{m_{j+1}}(x)\in C_{m_j}'\subset C_{m_{j-1}},\\
&g_{m_{j-1}}g_{m_j}r_{m_{j+1}}(x)\in C_{m_{j-1}}'\subset C_{m_{j-2}},\\
&\cdots\cdots\cdots\\
& g_{m_1}g_{m_2}\cdots g_{m_j}r_{m_{j+1}}(x)\in C_{m_1}',
\end{eqnarray*}
which imply that the arcs
$$[r(x), r_{m_{j+1}}(x)], ~g_{m_j}[r(x), r_{m_{j+1}}(x)],~ \cdots, ~ g_{m_1}g_{m_2}\cdots g_{m_j}[r(x), r_{m_{j+1}}(x)]$$
are pairwise disjoint. In particular,
$$r(x), ~ g_{m_j}r(x),~ g_{m_{j-1}}g_{m_j}r(x),~\cdots, ~g_{m_1}g_{m_2}\cdots g_{m_j}r(x)$$
 are pairwise distinct. By the arbitrariness of $j$, we see that
 the orbit of $r(x)$ is infinite.
 \medskip

Noting that any sequence of mutually disjoint subcontinua  of a dendrite forms a null sequence (see \cite[V.2.6]{Why}), the contractibility of $r$ follows from Claim 2.

\end{proof}

\section{Examples and questions}\label{examples}

As supplements to the main theorems, we give some examples in this section.

\medskip

\begin{exa}  For each $i\not=0\in \mathbb Z$, let $\theta_i={\rm sgn}(i)(1-\frac{1}{2|i|})\pi$; and let
 $I_i$ be the arc in the complex plane defined by $I_i=\{re^{{\rm i}\theta_i}: 0\leq r\leq \frac{1}{|i|}\}$.
 Let $X=\cup_{i\in\mathbb Z\setminus\{0\}}I_i$ and take the subspace topology of $\mathbb C$. Then $X$ is a
dendrite of infinite order. If $G$ is a countably infinite group, then it can act on the end point set of
$X$ transitively and freely. We then extend this action to $X$ by letting $g$ map $I_i$ affinely
to $I_{g(i)}$ for each $g\in G$ and $i\not=0\in \mathbb Z$. Clearly, the extended action is almost free. So, the condition
``with no infinite order points" in Theorem \ref{main theorem 1} cannot be removed.
\end{exa}

The following example indicates that the exact analogy to the Zimmer's rigidity for higher rank lattice actions on the circle
does not hold for higher rank lattice actions on dendrites of finite order.

\begin{exa} \label{almost finite}
Let $\Gamma=SL_{n}(\mathbb{Z})$ with $n\geq 3$. Fix a prime number $p$. For each positive integer $\alpha$, let
 $ \phi_{\alpha}: SL_{n}(\mathbb{Z})\rightarrow SL_{n}(\mathbb{Z}/p^{\alpha}\mathbb{Z})$
 denote the canonical homomorphism and let $\Gamma_{\alpha}=\ker(\phi_{\alpha}).$
Then $\Gamma_{\alpha}$ is a finite index subgroup of $\Gamma$, the so called {\it principal congruence subgroup} of $\Gamma$.
It is known that for each $\alpha$ the index
\[ [\Gamma_{\alpha}: \Gamma_{\alpha+1}]=p^{n^2-1} .\]
Set $\Gamma_0=\Gamma$. The sequence $(\Gamma_{\alpha})_{\alpha=0}^{\infty}$ forms a {\it group chain } of $\Gamma$:
\[ \Gamma_0>\Gamma_1>\Gamma_2>\cdots .\]
Thus we get a sequence of finite sets $\{\Gamma/\Gamma_{\alpha}: \alpha=0, 1, 2, \cdots\}$ on which $G$ acts by left translations.

\medskip
Now we associate to each $\alpha$ a finite combinatorial tree $Y_\alpha$ whose end point set is  $\Gamma/\Gamma_{\alpha}$,
and a finite action of $G$ on $Y_\alpha$ whose restriction to $\Gamma/\Gamma_{\alpha}$ coincides with the left translation action.
Precisely, let $Y_0=\{V_0\}={\Gamma/\Gamma_0}$ and let $\Gamma$ act on it trivially. Assume that for each $0\leq \beta\leq \alpha$,
$Y_\beta$ and the finite action of $G$ on it is defined. We let $Y_{\alpha+1}$ be the union of $Y_\alpha$ and the
set of  edges:
\[\{(\gamma\Gamma_{\alpha},  \gamma\gamma_1\Gamma_{\alpha+1}),(\gamma\Gamma_{\alpha},  \gamma\gamma_2\Gamma_{\alpha+1} ),\cdots, (\gamma\Gamma_{\alpha},  \gamma\gamma_{p^{n^{2}-1}}\Gamma_{\alpha+1}): \gamma\Gamma_{\alpha}\in\Gamma/\Gamma_{\alpha}\}, \]
where $\gamma_1,\cdots,\gamma_{p^{n^2-1}}$ is a set of coset representatives of $\Gamma_{\alpha+1}$ in $\Gamma_{\alpha}$.
Then we extend the action of $G$ from $Y_{\alpha}$ to $Y_{\alpha+1}$ by letting $g.\gamma\Gamma_{\alpha+1}=(g\gamma)\Gamma_{\alpha+1}$
for any $g, \gamma\in\Gamma$.

\medskip

For each $\alpha$, let $T_\alpha$ be the geometric realization of  $Y_\alpha$, which is a finite topological tree.
Then $G$ induces a finite action on each $T_\alpha$ in a canonical way.
 Let $\psi_{\alpha}: T_{\alpha+1}\rightarrow T_{\alpha}$ be a continuous surjective map defined by
\begin{itemize}
\item[(1)]  $\psi_{\alpha}|_{T_{\alpha}}={\rm id_{T_{\alpha}}}$;
\item[(2)] for each arc $[u,v]$ with $u\in V_{\alpha}, v\in V_{\alpha+1}$, $\psi_{\alpha}$ maps the whole arc $[u,v]$ to $u$.
\end{itemize}
From the definition,  we see that $\psi_{\alpha}(gx)=g\psi_{\alpha}(x)$ for  $x\in T_\alpha$ and $g\in \Gamma$. Thus
we get an inverse system $\{(T_\alpha, \Gamma): \alpha=0, 1, 2,\cdots\}$ with bonding maps $\psi_\alpha$.
Since each $\psi_\alpha$ is  monotone and onto,  $\underset{\longleftarrow}{\lim}(T_\alpha, \Gamma)$
is a dendrite by \cite[Theorem 10.36]{Nad}, which is of finite order by the construction. Clearly, the inverse limit
 $(\underset{\longleftarrow}{\lim}(T_\alpha, \Gamma), \Gamma)$ is almost finite.

\end{exa}

The following example shows that there does exist a non-almost-finite action by $SL_n(\mathbb Z)$ with $n\geq 3$ on a
dendrite of finite order.

\begin{exa} \label{not almost finite}
Let $(X, \Gamma)$ be the system constructed in Example \ref{almost finite}. Since dendrites are planar continua, we may assume that
$X$ is contained in $\mathbb R^2\times \{0\}\subset\mathbb R^3$. Fix an end point $e$ of $X$ and label its orbit $\Gamma e$ as $\{e_i: i=1,2,\cdots\}$.
Suppose the coordinate of $e_i$ is $(x_i, y_i, 0)$. Let $I_i=\{(x_i, y_i, t):0\leq t\leq 1/i\}$ for each $i$,
and let $Z=X\cup (\cup_{i=1}^\infty I_i)$. Then $Z$ is a dendrite of finite order contained in $\mathbb R^3$.
Extend the action of $\Gamma$ from $X$ to $Z$ by letting $g$ map $I_i$ to $I_j$ affinely if $e_j=ge_i$, for each $g\in\Gamma$.
Then we get an action $(Z, \Gamma)$, which is  not almost finite.
\end{exa}

We suspect that the action in Example \ref{not almost finite} is the only possible model for higher rank lattice actions on dendrites without infinite order points, and Theorem \ref{main theorem 2} has already supported this point of view. We end the paper with the following

\begin{ques}
Let $\Gamma$ be a higher rank lattice acting on a nondegenerate dendrite $X$ with no infinite order points. Does there exist a nondegenerate $\Gamma$-invariant subdendrite $Y$ satisfying the following items?

(1) There is   an inverse system of finite actions $\{(Y_i, G):i=1,2,3,\cdots\}$ with monotone bonding maps $\phi_i: Y_{i+1}\rightarrow Y_i$ and with each $Y_i$ being a dendrite, such that $(Y, G|Y)$ is topologically conjugate to the inverse limit $(\underset{\longleftarrow}{\lim}(Y_i, \Gamma), \Gamma)$.

(2) The first point map $r:X\rightarrow Y$ is a factor map from $(X, \Gamma)$ to $(Y, \Gamma|Y)$; if $x\in X\setminus Y$, then $r(x)$ is an end point of $Y$ with infinite orbit; for each $y\in Y$, $r^{-1}(y)$ is contractible, that is there is a sequence $g_i\in G$ with ${\rm diam}(g_ir^{-1}(y))\rightarrow 0$.
\end{ques}


\end{document}